\documentclass{amsart}
\usepackage{amssymb}
\usepackage{mathrsfs}
\usepackage[all]{xy}

\newcommand{\cL}{\mathcal{L}}
\newcommand{\Z}{\mathbb{Z}}

\newcommand{\mH}{\mathbb{H}}
\newcommand{\mC}{\mathbb{C}}
\newcommand{\mZ}{\mathbb{Z}}

\newcommand{\fR}{\mathfrak{R}}
\newcommand{\oR}{\mathfrak{R}_0}
\newcommand{\cR}{\hat{\mathfrak{R}}}
\newcommand{\coR}{\hat{\mathfrak{R}}_0}

\newcommand{\cA}{\mathcal{A}}
\newcommand{\cB}{\mathcal{B}}
\newcommand{\cC}{\mathcal{C}}
\newcommand{\cS}{\mathscr{S}}
\newcommand{\Kh}{\hat K}
\newcommand{\Db}{D^{\mathrm{b}}}
\newcommand{\Dm}{D^{-}}
\newcommand{\Dp}{D^{+}}
\newcommand{\Ddf}{D^\triangledown}
\newcommand{\Kpf}{K_{\mathrm{pf}}}
\newcommand{\Khpf}{\hat K_{\mathrm{pf}}}
\newcommand{\KH}{\mathbf{K}}

\DeclareMathOperator{\wt}{wt}
\DeclareMathOperator{\gmod}{gmod}
\newcommand{\mgmod}{\text{\textrm{-}}\mathrm{gmod}}

\newcommand{\simto}{\overset{\sim}{\to}}
\newcommand{\op}{\operatorname}
\newcommand{\END}{\op{End}}
\newcommand{\HOM}{\op{Hom}}

\DeclareMathOperator{\Ext}{Ext}
\DeclareMathOperator{\Irr}{Irr}
\DeclareMathOperator{\Proj}{Proj}

\DeclareMathOperator{\rad}{rad}

\numberwithin{equation}{section}
\newtheorem{thm}{Theorem}[section]
\newtheorem{lem}[thm]{Lemma}
\newtheorem{prop}[thm]{Proposition}

\theoremstyle{definition}
\newtheorem{defn}[thm]{Definition}
\newtheorem{ex}[thm]{Example}
\theoremstyle{remark}
\newtheorem{rmk}[thm]{Remark}

\title{Completions of Grothendieck Groups}

\author{Pramod N. Achar}

\author{Catharina Stroppel}

\thanks{The first author received support from NSF Grant No.~DMS-1001594.}

\begin{document}

\begin{abstract}
For a certain class of abelian categories, we show how to make sense of the ``Euler characteristic'' of an infinite projective resolution (or, more generally, certain chain complexes that are only bounded above), by passing to a suitable completion of the Grothendieck group.  We also show that right-exact functors (or their left-derived functors) induce continuous homomorphisms of these completed Grothendieck groups, and we discuss examples and applications coming from categorification.
\end{abstract}

\maketitle

\section{Introduction}
\label{sect:intro}

Let $\cA$ be a noetherian and artinian abelian category with enough projectives, and let $\Db(\cA)$ be its bounded derived category.  The inclusion $\cA \to \Db(\cA)$ gives rise to a natural isomorphism of Grothendieck groups
\begin{equation}\label{eqn:der-groth}
K(\cA) \simto K(\Db(\cA)).
\end{equation}
When $\cA$ has finite cohomological dimension, $K(\cA)$ captures a great deal of information about ``derived'' phenomena.  For instance, for any $X \in \cA$, we have
\begin{equation}\label{eqn:proj-sum}
[X] = \sum (-1)^i [P^i], \qquad
\text{where $P^\bullet \to X$ is a projective resolution.}
\end{equation}
If $\cB$ is another such category, then for any right-exact functor $F: \cA \to \cB$, the derived functor $\cL F$ induces a group homomorphism
\begin{equation}\label{eqn:der-homom}
[\cL F]: K(\cA) \to K(\cB).
\end{equation}
On the other hand, if $\cA$ has infinite projective dimension, we should replace $\Db(\cA)$ by $\Dm(\cA)$, the bounded above derived category,
but then its Grothendieck group cannot see anything at all, as it is zero by an ``Eilenberg swindle''-type argument; see~\cite{miy:gguc}.

However, when $\cA$ and $\cB$ are mixed categories with a Tate twist, a version of the statements \eqref{eqn:der-groth}, \eqref{eqn:proj-sum}, and~\eqref{eqn:der-homom} can be recovered.

We explain in this note how to replace $\Dm(\cA)$ by a certain subcategory $\Ddf(\cA)$ that is still large enough for derived functors, but small enough that a ``topological'' version of~\eqref{eqn:proj-sum} holds.  For these topological Grothendieck groups, certain infinite sums like~\eqref{eqn:proj-sum} converge, and derived functors give rise to \emph{continuous} homomorphisms.

More precisely, in this setting, the Grothendieck group $K(\cA)$ is naturally a module over the ring $\fR = \Z[q,q^{-1}]$.  It admits a completion $\Kh(\cA)$ that is a module over $\cR = \Z[[q]][q^{-1}]$.  The main results of the paper are summarized below.  (Further definitions and notation are given in Section~\ref{sect:notation}.)

\begin{thm}\label{thm:main}
Let $\cA$ be a noetherian and artinian mixed abelian category with enough projectives and a Tate twist.
\begin{enumerate}
\item The topological Grothendieck group $\KH(\Ddf(\cA))$ is a complete topological $\fR$-module.  Moreover, the natural map $K(\cA) \to \KH(\Ddf(\cA))$ is injective and induces an isomorphism\label{it:compl-isom}
\[
\Kh(\cA) \simto \KH(\Ddf(\cA)).
\]
\item Every object $X \in \Ddf(\cA)$ admits a projective resolution $P^\bullet$ with asymptotically decreasing weights.  In $\Kh(\cA)$, we have convergent series\label{it:conv-series}
\[
[X] = \sum_{i \in \Z} (-1)^i [H^i(X)] = \sum_{i \in \Z} (-1)^i [P^i].
\]
\item Both $\Irr(\cA_0)$ and $\Proj(\cA)_0$ span dense free $\cR$-submodules of $\Kh(\cA)$.  If those sets are finite, they each give an $\cR$-basis for $\Kh(\cA)$.\label{it:basis}
\item Let $\cB$ be another finite-length mixed category with a Tate twist, and let $F: \cA \to \cB$ be a right-exact functor that commutes with the resprective Tate twist.  If $F$ has finite weight amplitude, then $\cL F$ induces a continuous homomorphism of $\cR$-modules\label{it:der-functor}
\[
[\cL F]: \Kh(\cA) \to \Kh(\cB).
\]
\end{enumerate}
\end{thm}

The idea of completing Grothendieck groups arises from the concept of {\it categorification}, see e.g. \cite{KMS, Maz:lec}. There, $\mZ[q,q^{-1}]$-modules get realized as Grothendieck groups $K(\cA)$ of appropriately chosen graded categories $\cA$. The action of $q$ and $q^{-1}$ arises from shifting the grading (up and down). Often the  $\mZ[q,q^{-1}]$-modules in question come along with standard and canonical bases which then correspond to distinguished bases of $K(\cA)$. So far, representation theorists focused on the cases where the entries of the transformation matrices between the different bases were elements of $\mZ[q,q^{-1}]$. These numbers are then usually interpreted as Jordan--H\"older multiplicities or graded decomposition numbers. However, the theory of Lusztig's canonical bases or Kashiwara's crystal bases gives plenty of examples where the entries of the transformation matrix are contained only in the completion $\mZ[[q]][q^{-1}]$ of $\mZ[q,q^{-1}]$. Theorem \ref{thm:main} provides a possible categorical setup to handle such situations and could be viewed as the abstract context for categorifications of, for instance, modules for quantum groups. Although this  paper focuses on the abstract setup, there are already concrete examples known, for instance
in the context of categorification of Reshetikhin--Turaev--Viro invariants of links and $3$-manifolds; see~\cite{FSS}.

Following some set-up in Section~\ref{sect:notation}, the main theorem will be proved in Sections~\ref{sect:groth} and~\ref{sect:projres}.  Some examples and applications are indicated in Section~\ref{sect:example}.

\subsubsection*{Acknowledgements}
The authors are grateful to Sabin~Cautis for pointing out a flaw in an earlier version of this paper, and to Olaf Schn\"urer and Amnon Neeman for a number of very useful remarks on a previous draft.

\section{Notation and Definitions}
\label{sect:notation}

\subsection{Mixed (abelian) categories}

All abelian categories will be assumed to be \emph{finite-length categories} (i.e. noetherian and artinian) and to be skeletally small.  Let $\cA$ be an abelian category. We denote by $\Irr(\cA)$ the set of isomorphism classes of simple objects.  In this setting, the Grothendieck group $K(\cA)$ is a free abelian group on the set $\Irr(\cA)$.
Recall that $K(\cA)=F(\cA)/R(\cA)$, where $F(\cA)$ is the free abelian group on isomorphism classes $[M]$ of objects $M\in\cA$, and $R(\cA)$ is the subgroup generated by the expressions $[A]-[C]+[B]$ whenever there is a short exact sequence of the form $0\rightarrow A\rightarrow C\rightarrow B\rightarrow 0$. So the above claim follows from the existence of a Jordan--H\"older series and the uniqueness of its multiset of subquotients and states that
\begin{equation}\label{eqn:groth-free}
K(\cA) = \Z[\Irr(\cA)].
\end{equation}
Recall that $\cA$ is said to be a \emph{mixed} category if there is a function $\wt: \Irr(\cA) \to \Z$, called the \emph{weight function} such that
\begin{eqnarray}
\label{mixed}
\Ext^1(L,L') = 0 \qquad\text{if $\wt([L])\le \wt([L'])$ for simple objects $L$, $L'$.}
\end{eqnarray}
In the following we mostly write  $\wt(L)$ for $\wt([L])$.
A \emph{Tate twist} on a mixed category $\cA$ is an automorphism $(1): \cA \to \cA$ such that
\[
\wt(L(1)) = \wt(L) - 1
\qquad\text{if $L$ is simple.}
\]
Henceforth, all abelian categories will be mixed and equipped with a Tate twist.  For more details on mixed categories with Tate twist we refer to \cite{BGS, Del, Saito}.

\subsection{Weight filtration}

Recall the following standard fact (\cite[Lemma 4.1.2]{BGS}):

\begin{lem}
\label{lem:weightfilt}
Let $\cA$ be a mixed abelian category. Then any object $M\in\cA$  has a unique finite filtration $W_\bullet=W_\bullet(M)$ such that  $W_i/W_{i-1}$  is a direct sum of simple objects, all of weight $i$.
\end{lem}

This filtration is called the {\it weight filtration}. In case $W_i/W_{i-1}\not=0$ we say $i$ {\it
occurs as a weight} in $M$. If only one weight occurs, then $M$ is called {\it pure} of this weight.  In general, the maximal weight occurring in $M$ is called the {\em degree} of $M$. We say $M$ has weights $\leq n$ if the degree of $M$ is smaller or equal $n$.

Note that morphisms are strictly compatible with weight filtrations, in the sense that $f(W_i(M))\subset W_i(N)$ for any morphism $f:M\rightarrow N$ between objects $M$, $N$ in $\cA$ with weight filtrations $W_\bullet(M)$ and $W_\bullet(N)$. It is a consequence that
\begin{equation}\label{eqn:mixed-defn}
\Ext^i(L, L') = 0, \qquad\text{if $\wt([L]) -i < \wt([L'])$,}
\end{equation}
or, equivalently,
\begin{eqnarray}\label{eqn:mixed-defn2}
\Ext^i(L, L') \not= 0 &\Rightarrow&\wt([L])\geq\wt([L'])+i.
\end{eqnarray}

\begin{ex}
\label{ex:mixed}
Our standard example of a mixed abelian category is the category  $A\mgmod$ of finite-dimensional graded modules over a finite-dimensional positively graded algebra $A=\oplus_{i\in\mathbb{Z}_{\geq0}} A_i$ over the complex numbers with semisimple $A_0$. Each simple module $L$ is concentrated in a single degree $-\wt(L)$ and the Tate twist $(1)$ is given by the automorphism $\langle 1\rangle$ which shifts the degree up by $1$, i.e. if $M=\oplus_{j\in\mZ}M_j$ then $M(1)=M\langle 1\rangle$ is the graded module with graded components $(M\langle 1\rangle))_j=M_{j-1}$.
Any $\mC$-linear mixed abelian category can be realized as the category of modules over a projective limit of such positively graded algebras; see \cite[4.1.6]{BGS} for a precise statement.
\end{ex}

For $n \in \Z$, let $\cA_{\le n}$ (resp.~$\cA_n$, $\cA_{\ge n}$), be the Serre subcategory of $\cA$ generated by the simple objects of weight${}\le n$ (resp.~$n$,${}\ge n$).  If $m \le n$, we also put $\cA_{[m,n]} = \cA_{\ge m} \cap \cA_{\le n}$.  For any $X \in \cA$, the weight filtration (Lemma \ref{lem:weightfilt}) defines a functorial short exact sequence
\[
0 \to \beta_{\le n}X \to X \to \beta_{\ge n+1}X \to 0
\]
where $\beta_{\le n}X$ has weights${}\le n$ and $\beta_{\ge n+1}X$ has weights${}\ge n+1$.  Moreover, the functors $\beta_{\le n}: \cA \to \cA_{\le n}$ and $\beta_{\ge n+1}: \cA \to \cA_{\ge n+1}$ are exact, so we can apply them to a chain complex $C^\bullet$ in $\cA$ and get a short exact sequence of chain complexes.  These functors induce derived functors $\Dm(\cA) \to \Dm(\cA)$, so for any object $X \in \Dm(\cA)$, there is a functorial distinguished triangle
\begin{eqnarray}
\label{triangle}
\beta_{\le n}X \to X \to \beta_{\ge n+1}X \to
\end{eqnarray}
in $\Dm(\cA)$.  (The same remarks apply to the bounded $\Db(\cA)$ and bounded below $\Dp(\cA)$ derived categories as well, but we will work primarily with $\Dm(\cA)$.)  These functors endow $\Dm(\cA)$ with a \emph{baric structure} in the sense of~\cite{at:bstc}. The connecting homomorphism from \eqref{triangle} is in fact unique (in contrast with the  context of weight structures as studied e.g. in \cite{Bondarko}, \cite{Schnuerer}).

\begin{defn}
Let $F: \cA \to \cB$ be an additive functor between two mixed abelian categories.  The \emph{weight amplitude} of $F$ is defined to be the infimum of the set
\[
\left\{ a \in \Z_{\ge 0} \mid \text{$F(\cA_{\le n}) \subset \cB_{\le n+a}$ for all $n \in \Z$} \right\} \cup \{ +\infty\}.
\]
\end{defn}

\subsection{Coefficients rings and Grothendieck groups}
Most Grothendieck groups we consider will naturally be modules over one of the following rings (two of which were mentioned in Section~\ref{sect:intro}):
\[
\oR = \Z[q],
\qquad
\fR = \Z[q,q^{-1}],
\qquad
\coR = \Z[[q]],
\qquad
\cR = \fR \otimes_{\oR} \coR = \Z[[q]][q^{-1}].
\]
For instance, the Tate twist induces an automorphism $q: K(\cA) \to K(\cA)$, where $[X(1)] = q[X]$, and so makes $K(\cA)$ into an $\fR$-module.  It also restricts to a fully faithful, exact functor $(1): \cA_{\le n} \to \cA_{\le n}$, but this is no longer an equivalence.  The Grothendieck group $K(\cA_{\le n})$ is naturally an $\oR$-submodule of $K(\cA)$.  It follows from~\eqref{eqn:groth-free} that $K(\cA)$ is free as an $\fR$-module (see \cite[Lemma 4.3.2]{BGS}).

In fact, for any $n \in \Z$, we have canonical isomorphisms
\begin{equation}\label{eqn:groth-rfree}
K(\cA_{\le n}) \cong \oR[\Irr(\cA_n)]
\qquad\text{and}\qquad
K(\cA) \cong \fR \otimes_{\oR} K(\cA_{\le n}).
\end{equation}
The $\oR$-module $K(\cA_{\le n})$ is equipped with a natural $(q)$-adic topology, in which the submodules
\[
q^i \cdot K(\cA_{\le n}) = K(\cA_{\le n-i})
\]
for $i\geq 0$ constitute a basis of neighborhoods around $0$.  Similarly, we endow $K(\cA)$ with a topology (also called ``$(q)$-adic'') by declaring the submodules $K(\cA_{\le i}) \subset K(\cA)$ to be a basis of neighborhoods around $0$.  It follows from~\eqref{eqn:groth-rfree} that
\begin{equation}\label{eqn:groth-hausdorff}
\bigcap_{m \in \Z} K(\cA_{\le m}) = 0.
\end{equation}
In other words, the $(q)$-adic topology on $K(\cA)$ or $K(\cA_{\le n})$ is Hausdorff.  Let
\[
\Kh(\cA_{\le n})
\qquad\text{and}\qquad
\Kh(\cA)
\]
denote the completions of each of these modules in the $(q)$-adic topology.  These completions are modules over $\coR$ and $\cR$, respectively.

\subsection{Definition of $\Ddf(\cA)$}

Given a mixed abelian category with a Tate twist, we define the following full subcategory of $ \Dm(\cA)$:
\[
\Ddf(\cA) = \left\{ X \in \Dm(\cA) \,\Big|\,
\begin{array}{c}
\text{for each $m \in \Z$, only finitely many of the $H^i(X)$} \\
\text{contain a composition factor of weight${}>m$}
\end{array} \right\}.
\]
It is easy to see that  $\Ddf(\cA)$ is closed under suspensions (or shifts) $[i]$, $i\in\mathbb{Z}$, and cones, and hence that it is a full triangulated subcategory of $\Dm(\cA)$.

For $n \in \Z$, we also define the following full subcategories of $\Ddf(\cA)$:
\begin{align*}
\Ddf_{\le n}(\cA) &= \{ X \in \Ddf(\cA) \mid \text{for all $i \in \Z$, $H^i(X)$ has weights${}\le n$} \}, \\
\Ddf_{\ge n}(\cA) &= \{ X \in \Ddf(\cA) \mid \text{for all $i \in \Z$, $H^i(X)$ has weights${}\ge n$} \}.
\end{align*}
They are triangulated categories. If $m \le n$, we also put $\Ddf_{[m,n]}(\cA) = \Ddf_{\ge m}(\cA) \cap \Ddf_{\le n}(\cA)$.  It follows from the definition of $\Ddf(\cA)$ that any object in $\Ddf_{\ge n}(\cA)$ has only finitely many nonzero cohomology objects, so
\begin{equation}
\label{eqn:DninD}
\Ddf_{\ge n}(\cA) \subset \Db(\cA).
\end{equation}
The Tate twist induces an autoequivalence $(1): \Ddf(\cA) \to \Ddf(\cA)$ and a fully faithful functor $(1): \Ddf_{\le n}(\cA) \to \Ddf_{\le n}(\cA)$, so $K(\Ddf(\cA))$ and $K(\Ddf_{\le n}(\cA))$ are modules over $\fR$ and $\oR$, respectively. The categories $\Ddf_{\geq m}(\cA)$ are not preserved by the Tate twist, but nevertheless we will construct in the next section an $\oR$-module structure on their Grothendieck groups.

\subsection{Topological Grothendieck groups}
Recall that the Grothen\-dieck group of a small triangulated category $\cC$ is defined as $K(\cC)=F(\cC)/R(\cC)$, where $F(\cC)$ is the free abelian group on isomorphism classes $[M]$ of objects $M\in\cC$, and $R(\cC)$ is the ideal generated by the expressions $[A]-[C]+[B]$ whenever there is a distinguished triangle of the form $A\rightarrow C\rightarrow B\rightarrow A[1]$.  Suppose now that $\cC$ is a subcategory of the derived category of $\cA$ that is stable under $\beta_{\le m}$ and $\beta_{\ge m}$ for all $m$.  Let
\[
I(\cC) = \{ f \in K(\cC) \mid \text{$[\beta_{\ge m}]f = 0$ in $K(\cC)$ for all $m \in \Z$} \}.
\]
We define the \emph{topological Grothendieck group} of $\cC$ to be
\[
\KH(\cC) = K(\cC) / I(\cC).
\]
The reason for the terminology will become clear in Remark~\ref{rmk:hausdorff}.  We will eventually prove a general result (cf.~Theorem~\ref{thm:main}\eqref{it:der-functor}) about derived functors and the topological Grothendieck group.  For now, let us note simply that any functor of triangulated categories that commutes with all $\beta_{\ge m}$  induces a homomorphism of topological Grothendieck groups.

\begin{rmk}
If $\cC = \Db(\cA)$, the topological Grothendieck group coincides with the ordinary Grothendieck group.  Indeed, it can be shown that $K(\cA) \cong K(\cC)$, and then, in view of~\eqref{eqn:groth-free} and~\eqref{eqn:groth-rfree}, it follows from~\eqref{eqn:groth-hausdorff} that $I(\cC) = 0$.  Therefore, $K(\Db(\cA)) \cong \KH(\Db(\cA))$.
\end{rmk}

\section{The Grothendieck Group of $\Ddf(\cA)$}
\label{sect:groth}

The main goal of this section is to prove part~\eqref{it:compl-isom} of Theorem~\ref{thm:main}.

\subsection{Sequences of $\oR$-modules}
\label{subsect:ksum}

The categories $\cA_{\ge m}$ and $\Ddf_{\ge m}(\cA)$ are not preserved by the Tate twist $(1)$, so we use a different functor to make $K(\cA_{\ge m})$ and $K(\Ddf_{\ge m}(\cA))$ into $\oR$-modules: we put
\[
q \cdot [X] = [\beta_{\ge m}(X(1))] \qquad\text{for $X \in \cA_{\ge m}$ or $X \in \Ddf_{\ge m}(\cA)$.}
\]
This definition makes sense because $\beta_{\ge m} \circ (1)$ is an exact functor that preserves $\cA_{\ge m}$ and $\Ddf_{\ge m}(\cA)$.  The same definition also makes sense for $\cA_{[m,n]}$ and $\Ddf_{[m,n]}(\cA)$.

\begin{lem}\label{lem:ksum}
For any $n \in \Z$, there is commutative diagram of $\oR$-modules
\begin{equation}\label{eqn:ksum}
\vcenter{\xymatrix{
0 \ar[r] & K(\cA_{\le n}) \ar[r]\ar[d] & K(\cA) \ar[r]^-{[\beta_{\ge n+1}]}\ar[d] & K(\cA_{\ge n+1}) \ar[r]\ar[d] & 0 \\
0 \ar[r] & \KH(\Ddf_{\le n}(\cA)) \ar[r] & \KH(\Ddf(\cA)) \ar[r]^-{[\beta_{\ge n+1}]} & K(\Ddf_{\ge n+1}(\cA)) \ar[r] & 0}}
\end{equation}
in which the rows are short exact sequences.  Moreover, the first two vertical maps are injective, and the last one is an isomorphism.
\end{lem}
\begin{proof}
We begin by treating the second row in this diagram.  Consider the surjective map $\gamma: F(\Ddf(\cA)) \to F(\Ddf_{\le n}(\cA)) \oplus F(\Ddf_{\ge n+1}(\cA))$ defined as $[X] \mapsto ([\beta_{\le n}X], [\beta_{\ge n+1}X])$.  Given a distinguished triangle $A \to X \to B\to$ in $\Ddf(\cA)$, we have
\[
\gamma([X] - [A] - [B]) =  ([\beta_{\le n}X] - [\beta_{\le n}A] - [\beta_{\le n}B],
[\beta_{\ge n+1}X] - [\beta_{\ge n+1}A] - [\beta_{\ge n+1}B]).
\]
Since the functors $\beta_{\le n}$ and $\beta_{\ge n+1}$  are functors of triangulated categories, this calculation shows that $\gamma(R(\Ddf(\cA))) \subset R(\Ddf_{\le n}(\cA)) \oplus R(\Ddf_{\ge n+1}(\cA))$.  Since the restriction of $\beta_{\le n}$, resp.~$\beta_{\ge n+1}$, to $\Ddf_{\le n}(\cA)$, resp.~$\Ddf_{\ge n+1}(\cA)$, is the identity functor, we actually have $\gamma(R(\Ddf(\cA))) = R(\Ddf_{\le n}(\cA)) \oplus R(\Ddf_{\ge n+1}(\cA))$.  We conclude that $\gamma$ induces an isomorphism of abelian groups
\[
K(\Ddf(\cA)) = K(\Ddf_{\le n}(\cA)) \oplus K(\Ddf_{\ge n+1}(\cA)).
\]
Because $\beta_{\ge m}$ preserves each of $\Ddf_{\le n}(\cA)$ and $\Ddf_{\ge n+1}(\cA)$, we have $I(\Ddf(\cA)) = I(\Ddf_{\le n}(\cA)) \oplus I(\Ddf_{\ge n+1}(\cA))$.  On the other hand, the fact that $\beta_{\ge n+1}$ is the identity functor on $\Ddf_{\ge n+1}(\cA)$ implies that $I(\Ddf_{\ge n+1}(\cA)) = 0$.  We deduce that
\begin{equation}\label{eqn:ksum-ab}
\KH(\Ddf(\cA)) = \KH(\Ddf_{\le n}(\cA)) \oplus K(\Ddf_{\ge n+1}(\cA)).
\end{equation}
In this direct sum, the inclusion $\KH(\Ddf_{\le n}(\cA)) \to \KH(\Ddf(\cA))$ is in fact induced by the inclusion functor $\iota: \Ddf_{\le n}(\cA) \to \Ddf(\cA)$, since $\beta_{\le n} \circ \iota = \mathrm{id}$ and $\beta_{\ge n+1} \circ \iota = 0$.  Since $\iota$ commutes with the Tate twist, the inclusion map $K(\Ddf_{\le n}(\cA)) \to K(\Ddf(\cA))$ is a homomorphism of $\oR$-modules.  On the other hand, the setup is such that $[\beta_{\ge n+1}]: K(\Ddf(\cA)) \to K(\Ddf_{\ge n+1}(\cA))$
commutes with the action of $q \in \oR$ on both groups.
Thus,~\eqref{eqn:ksum-ab} gives rise to the desired short exact sequence (the second line of \eqref{eqn:ksum}) of $\oR$-modules.

Since $\beta_{\le n}$ and $\beta_{\ge n+1}$ are $t$-exact (for the $t$-structure induced from the standard $t$-structure), the same argument can be repeated with the abelian categories $\cA_{\le n}$, $\cA$, and $\cA_{\ge n+1}$ (but skipping the passage to the topological Grothendieck group), yielding the $\coR$-module structure and exactness of the first row in the diagram~\eqref{eqn:ksum}.  Since all maps in that diagram are induced by inclusion functors or by $\beta_{\ge n+1}$, it is easy to see that the diagram commutes.

Recall from \eqref{eqn:DninD} that $\Ddf_{\ge n+1}(\cA) \subset \Db(\cA)$.  The fact that $K(\cA_{\ge n+1}) \to K(\Ddf_{\ge n+1}(\cA))$ is an isomorphism follows from the fact that $\cA_{\ge n+1}$ is the heart of a \emph{bounded} $t$-structure on $\Ddf_{\ge n+1}(\cA)$.

Next, consider an element $f \in K(\cA)$.  We may write $f = a_1[X_1] + \cdots + a_k[X_k]$ for suitable simple objects $X_i \in \cA$.  Now, choose $n$ such that $n < \wt(X_i)$ for all $i$.  It is clear from Lemma \ref{lem:weightfilt} that $[\beta_{\ge n+1}]f \ne 0$.  In view of the preceding paragraph, it follows from the commutativity of~\eqref{eqn:ksum} that the image of $f$ in $\KH(\Ddf(\cA))$ is nonzero.  Thus, the middle vertical arrow in~\eqref{eqn:ksum} is injective.  The injectivity of the first vertical arrow is then clear as well.
\end{proof}

The same reasoning yields the following related statement.

\begin{lem}\label{lem:ksum-fin}
Supose $m \le n$.  There is a commutative diagram of $\oR$-modules
\begin{equation}
\label{eqn:sestwo}
\vcenter{\xymatrix{
0 \ar[r] & K(\cA_{\le m}) \ar[r]\ar[d] & K(\cA_{\le n}) \ar[r]\ar[d] & K(\cA_{[m+1,n]}) \ar[r]\ar[d] & 0 \\
0 \ar[r] & \KH(\Ddf_{\le m}(\cA)) \ar[r] & \KH(\Ddf_{\le n}(\cA)) \ar[r] & K(\Ddf_{[m+1,n]}(\cA)) \ar[r] & 0}}
\end{equation}
in which the rows are short exact sequences.  Moreover, the first two vertical maps are injective, and the last is an isomorphism. \qed
\end{lem}

\subsection{$(q)$-adic topology}
Recall from \eqref{eqn:ksum} that $\KH(\Ddf_{\le m}(\cA))$ can naturally be identified with a $\oR$-submodule of $\KH(\Ddf(\cA))$ (or of $\KH(\Ddf_{\le n})$, if $m \le n$).  Thus, we are at last able to define the $(q)$-adic topology on these modules: we take the set of submodules of the form $\KH(\Ddf_{\le m})$ to be a basis of neighborhoods of $0$.  It follows from the proof of Lemma~\ref{lem:ksum} that for $s\in\mZ$
\begin{equation}
\label{eqn:qacts}
q^s \cdot \KH(\Ddf_{\le m}(\cA)) = \KH(\Ddf_{\le m-s}(\cA)),
\end{equation}
so $\KH(\Ddf_{\le n}(\cA))$ and $\KH(\Ddf(\cA))$ are naturally topological $\oR$- and $\fR$-modules, respectively.

\begin{lem}\label{lem:hausdorff}
In the $(q)$-adic topology, the groups $\KH(\Ddf_{\le n}(\cA))$ and $\KH(\Ddf(\cA))$ are Hausdorff.
\end{lem}
\begin{proof}
Being Hausdorff is equivalent to the condition that
\[
\bigcap_{m \in \Z} \KH(\Ddf_{\le m}(\cA)) = \{ 0\}.
\]
If $f \in \bigcap_{m \in \Z} \KH(\Ddf_{\le m}(\cA))$, it follows from Lemma~\ref{lem:ksum} and \eqref{eqn:groth-rfree} that $[\beta_{\ge m}]f = 0$ for all $m$, but then it is clear from the definition of $\KH(\Ddf(\cA))$ that $f = 0$.
\end{proof}

\begin{rmk}\label{rmk:hausdorff}
The statements of Section~\ref{subsect:ksum} are still true if we replace $\KH(\Ddf(\cA))$  by $K(\Ddf(\cA))$, and the definition of the $(q)$-adic topology makes sense for $K(\Ddf(\cA))$ as well, but the resulting space is not Hausdorff.  In fact, $\KH(\Ddf(\cA))$ is the \emph{universal} Hausdorff quotient of $K(\Ddf(\cA))$, in the sense that every continuous homomorphism from $K(\Ddf(\cA))$ to a Hausdorff abelian group factors through $\KH(\Ddf(\cA))$. This construction is well-known in the context of topological groups, see e.g \cite[Note after 3.22]{Strlccpt}.
\end{rmk}

\begin{lem}\label{lem:complete}
The $\oR$-module $\KH(\Ddf_{\le n}(\cA))$ is complete in the $(q)$-adic topology. Indeed, the natural map $K(\cA_{\le n}) \to \KH(\Ddf_{\le n}(\cA))$ induces an isomorphism
\[
\Kh(\cA_{\le n}) \to \KH(\Ddf_{\le n}(\cA)).
\]
\end{lem}
\begin{proof}
Since $q^s \cdot \KH(\Ddf_{\le n}(\cA)) = \KH(\Ddf_{\le n-s}(\cA))$ for any $s\in\mZ_{\geq0}$, it follows from Lemma~\ref{lem:ksum-fin}, with $m=n-s$,  that
\[
\KH(\Ddf_{\le n}(\cA)) / q^s \cdot \KH(\Ddf_{\le n}(\cA)) \cong K(\Ddf_{[n-s+1,n]}(\cA)).
\]
Suppose we have a sequence of elements $f_i\in K(\Ddf_{[n-i+1,n]}(\cA))$, $i\in\mZ_{\geq0}$ satisfying the condition that $[\beta_{\ge n-j+1}]f_i = f_j$ when $j < i$.  To show that $\KH(\Ddf_{\le n}(\cA))$ is complete, we must exhibit an element $g\in\KH(\Ddf_{\le n}(\cA))$ such that $[\beta_{\ge n-i+1}]g = f_i$ for all $i$.

By Lemma~\ref{lem:ksum-fin} again, we identify $K(\Ddf_{[n-i+1,n]}(\cA))$ with $K(\cA_{[n-i+1,n]})$, viewed as a subgroup of $K(\cA)$.  Regarding all the $f_i$ as elements of $K(\cA)$, we can form the elements
\[
a_i = f_i - f_{i-1} = f_i - [\beta_{\ge n-i+2}]f_i \in K(\cA_{n-i+1}).
\]
Then $f_i = a_1 + a_2 + \cdots + a_i$ for all $i$.  Since $K(\cA_{n-i+1})$ is the free abelian group on $\Irr(A_{n-i+1})$, we can write
\[
a_i = c_{i1}[L_{i1}] + \cdots + c_{i,r_i}[L_{i,r_i}] - d_{i1}[M_{i1}] - \cdots - d_{i,s_i}[M_{i,s_i}]
\]
for unique (up to renumbering) $[L_{ij}], [M_{ij}] \in \Irr(\cA_{n-i+1})$, and $c_{ij}, d_{ij} > 0$.  Now, let $X^\bullet$ be the chain complex with trivial differentials and
\begin{equation}
X^k=
\begin{cases}
0 &\text{if $k \ge 0$},\\
\bigoplus_{j=1}^{r_i} L_{ij}^{\oplus c_{ij}}& \text{if $k=-2i < 0$ is even},\\
\bigoplus_{j=1}^{s_i} M_{ij}^{\oplus d_{ij}}& \text{if $k=-2i+1 < 0$ is odd}.
\end{cases}
\end{equation}
By construction, $H^k(X^\bullet) \cong X^k$ vanishes for $k \ge 0$, and is pure of weight $n+1 + \lfloor k/2 \rfloor$ for $k < 0$, so $X^\bullet \in \Ddf_{\le n}(\cA)$.  It is easy to see that $[\beta_{\ge n-i+2} X^\bullet] = f_i$, so $g = [X^\bullet]$ is the element we were looking for.

Finally, we also see from Lemma~\ref{lem:ksum-fin} that
$$K(\cA_{\le n})/ q^i \cdot K(\cA_{\le n}) \cong \KH(\Ddf_{\le n}(\cA)) / q^i \cdot \KH(\Ddf_{\le n}(\cA))$$ for each $i$,  so $K(\cA_{\le n})$ and $\KH(\Ddf_{\le n}(\cA))$ have the same completion.
\end{proof}

\begin{proof}[Proof of Theorem~\ref{thm:main}\eqref{it:compl-isom}]
The injectivity of $K(\cA) \to \KH(\Ddf(\cA))$ was established in Lemma~\ref{lem:ksum}.  Since every Cauchy sequence in $K(\cA)$ or $\KH(\Ddf(\cA))$ is contained in some submodule $K(\cA_{\le n})$ or $\KH(\Ddf_{\le n}(\cA))$,  Lemma~\ref{lem:complete} implies that $\KH(\Ddf(\cA))$ is complete, and that $\Kh(\cA) \to \KH(\Ddf(\cA))$ is an isomorphism.
\end{proof}

\section{Projective Resolutions and Derived Functors}
\label{sect:projres}

We will prove the remaining parts of Theorem~\ref{thm:main} in this section.   Henceforth, $\cA$ is assumed to have enough projectives.  Since $\cA$ is also assumed to be a finite-length category, Fitting's lemma and its consequences hold; for instance, each projective is a direct sum of finitely many indecomposable ones.
The degree of an indecomposable projective $P$ is the integer
\[
\deg(P) = \wt(\text{the unique simple quotient of $P$}).
\]
For $n \in \Z$, let $\Proj(\cA)_n$ denote the set of isomorphism classes of indecomposable projectives of degree $n$.  Obviously, the map $P \mapsto P/\rad P$ induces a bijection
\[
\Proj(\cA)_n \simto \Irr(\cA_n).
\]
By considering the weight filtration (Lemma \ref{lem:weightfilt}, see also Example \ref{ex:mixed}), one can see that
\begin{equation}\label{eqn:wtbound}
\deg(P) = n \qquad\text{implies}\qquad P \in \cA_{\le n},
\end{equation}
with the convention $\deg(0)=-\infty$. More generally, the degree of a projective object is simply the maximum of the degrees of its indecomposable summands, and the degree of an arbitrary object is the degree of its projective cover.

\begin{defn}
A bounded-above complex $P^\bullet$ of projectives is said to have \emph{asymptotically decreasing weights} if for each $m \in \Z$, all but finitely many of the terms $P^i$ have degree ${}\le m$.
\end{defn}

\begin{lem}\label{lem:ddf-proj}
The following conditions on an object $X \in \Dm(\cA)$ are equivalent:
\begin{enumerate}
\item $X \in \Ddf_{\le n}(\cA)$.\label{it:ddf}
\item $X$ is quasi-isomorphic to a bounded-above complex of projectives $P^\bullet$ with asymptotically decreasing weights where each term $P^i$ is of degree ${}\le n$.\label{it:fproj}
\end{enumerate}
\end{lem}
\begin{proof}
In view of~\eqref{eqn:wtbound}, it is obvious that condition~\eqref{it:fproj} implies condition~\eqref{it:ddf}.  For the other implication, we first consider the special case where $X \in \cA$.  Let $d$ be the degree of $X$, and let $Q^\bullet$ be a minimal projective resolution of $X$.  The projective cover of a simple object $L$ occurs as a direct summand of $Q^i$
(for $i \le 0$) if and only if $\Ext^{-i}(X,L) \ne 0$.  This can only happen if $\wt(L) \le d+i$, so $Q^i$ is of degree ${}\le d+i$.  Using~\eqref{eqn:wtbound} again, we see that the complex $Q^\bullet$ satisfies condition~\eqref{it:fproj}.

For general $X$, choose a minimal projective resolution $Q_i^\bullet$ for each cohomology object $H^i(X)$.  Then $X$ is quasi-isomorphic to a complex $P^\bullet$ with terms of the form
\[
P^i = \bigoplus_{k \ge 0} Q_{i+k}^{-k}.
\]
Let $N$ be the largest integer such that $H^N(X) \ne 0$.  (Such an $N$ exists because $X$ is bounded above.) For $i \le N$, let $d_i$ be the degree of $H^i(X)$, or let $d_i = -\infty$ if $H^i(X) = 0$.  Next, let
\[
a_i = \max \left\{d_i, d_{i+1}-1, d_{i+2}-2, \ldots, d_N - (N-i) \right\}.
\]
Note that $P^i$ is of degree ${}\le a_i$.  In particular, each $P^i$ is of degree ${}\le n$.  Next, given $m \in \Z$, there is a $k_0$ such that $d_i \le m$ for all $i \le k_0$.  Let
\[
k = \min \{ i - d_i + m \mid k_0 \le i \le N \}.
\]
We claim that for all $i \le k$, $a_i \le m$.  Indeed, if $i \le k$ and $0 \le j \le N - i$, then
\begin{align*}
d_{i+j} - j &\le d_{i+j} \le m &&\text{if $i+j \le k_0$,} \\
d_{i+j} - j &= i+m - (i+j - d_{i+j} + m) \le i+m -k \le m &&\text{if $k_0 \le i+j \le N$.}
\end{align*}
Hence $P^\bullet$ is a bounded above complex of projectives with asymptotically decreasing weights, so it satisfies condition~\eqref{it:fproj}, as desired.
\end{proof}

\begin{lem}\label{lem:cont}
Let $F: \cA \to \cB$ be a right-exact functor commuting with the respective Tate twist.  If $F$ has weight amplitude $\alpha < \infty$, then the left-derived functor $\cL F: \Dm(\cA) \to \Dm(\cB)$ has the property that
\begin{equation}\label{eqn:lf-cont}
\cL F (\Ddf_{\le n}(\cA)) \subset \Ddf_{\le n+\alpha}(\cB).
\end{equation}
\end{lem}
\begin{proof}
Given $X \in \Ddf_{\le n}(\cA)$, choose a projective resolution $P^\bullet$ satisfying the condition in Lemma~\ref{lem:ddf-proj}\eqref{it:fproj}.  It is clear from~\eqref{eqn:wtbound} that all terms of the complex $F(P^\bullet)$ have weights${}\le n+\alpha$, and that for any $m$, only finitely terms have a composition factor of weight${}>m$.  The same then holds for its cohomology objects, so $\cL F(X) \in \Ddf_{\le n+\alpha}(\cB)$.
\end{proof}
We finish now this section by proving the remaining parts of Theorem~\ref{thm:main}.
\begin{proof}[Proof of Theorem~\ref{thm:main}\eqref{it:conv-series}] Given $m \in \Z$, let $k$ be such that for all $i \le k$, the cohomology $H^i(X)$ has weights ${}\le m$.
Then $[X] = [\tau^{\le k}X] + [\tau^{\ge k+1}X]$ where $\tau^{\le k}$, $\tau^{\geq k}$ denote the usual truncation functors in triangulated categories.
Since $X$ is bounded above, $\tau^{\ge k+1}X$ has only finitely many nonzero cohomology objects, and it is clear that $[\tau_{\ge k+1}X] = \sum_{i=k+1}^\infty (-1)^i[H^i(X)]$.  Moreover, by construction, $\tau^{\le k}X \in \Ddf_{\le m}(\cA)$, so
\[
[X] - \sum_{i=k+1}^\infty (-1)^i[H^i(X)] \in \KH(\Ddf_{\le m}(\cA)).
\]
Thus, the series $\sum (-1)^i [H^i(X)]$ converges to $[X]$.  The argument for $\sum (-1)^i[P^i]$ is similar.
\end{proof}

\begin{proof}[Proof of Theorem~\ref{thm:main}\eqref{it:basis}]
A description of the completion of a free module can be found in~\cite[\S 2.4]{sim:shpcm}.  It follows from that description that the basis of a free $\oR$-module spans a dense free $\coR$-submodule of its completion, and that the two coincide if the basis is finite.  Thus, it follows from~\eqref{eqn:groth-rfree} that $\Irr(\cA_0)$ spans a dense free $\coR$-submodule of $\Kh(\cA_{\le 0})$.

The case of $\Proj(\cA)_0$ is somewhat different, since this set does \emph{not} give an $\oR$-basis for $K(\cA_{\le 0})$ in general.  However, recall that if $P \in \Proj(\cA)_0$ and if $L \in \Irr(\cA_0)$ is its unique irreducible quotient, then in $K(\cA_{\le 0})$, we have
\[
[P] = [L] + (\text{terms in $q \cdot K(\cA_{\le 0})$}).
\]
It is easy to deduce from this that the elements of $\Proj(\cA)_0$ are linearly independent in $K(\cA_{\le 0})$: any relation would give rise to a relation among elements of $\Irr(\cA_0)$.  Thus, the $\oR$-submodule $\Kpf(\cA_{\le 0}) \subset K(\cA_{\le 0})$ generated by $\Proj(\cA)_0$ is free.  It follows that the corresponding $\fR$-submodule $\Kpf(\cA) \subset K(\cA)$ is free as well.  Since completion is left-exact, we have a natural inclusion $\Khpf(\cA) \subset \Kh(\cA)$.  The argument of the previous paragraph shows that $\Proj(\cA)_0$ spans a free dense submodule of $\Khpf(\cA)$, so it remains only to show that this submodule is also dense in $\Kh(\cA)$.  But this follows from the fact that the class of every object in $\Ddf(\cA)$ can be written as a convergent series of projectives.
\end{proof}

\begin{proof}[Proof of Theorem~\ref{thm:main}\eqref{it:der-functor}]
We see from Lemma~\ref{lem:cont} that $\cL F(\Ddf(\cA)) \subset \Ddf(\cB)$, so we certainly have an induced map $[\cL F]: K(\Ddf(\cA)) \to K(\Ddf(\cB))$.  Moreover, if $f \in K(\Ddf(\cA))$ is such that $[\beta_{\ge m}]f = 0$, it follows from Lemmas~\ref{lem:ksum} and~\ref{lem:cont} that $[\beta_{\ge m+\alpha}][\cL F]f = 0$.  In particular, if $f \in I(\cA)$, then $[\cL F]f \in I(\cB)$, so we actually have an induced map $[\cL F]: \KH(\Ddf(\cA)) \to \KH(\Ddf(\cB))$.  The assertion that it is continuous is then just a restatement of~\eqref{eqn:lf-cont}.
\end{proof}

\section{Examples and Applications}
\label{sect:example}

\subsection{Graded modules over a graded local ring}
Let $k$ be a field and $\mH=\bigoplus_{i\in \mZ_{\geq 0}}\mH^i$ a finite-dimensional positively graded connected (i.e., $\mH^0=k$) $k$-algebra. Then $\mH$ is graded local with maximal ideal $\mathfrak{m}=\bigoplus_{i\in \mZ_{>0}}\mH^i$ and has, up to isomorphism and grading shift, a unique irreducible (finite-dimensional) graded $\mH$-module, namely the trivial module $L=k=\mH^0$.
Let $\mH\mgmod$ be the category of finite dimensional $\mZ$-graded $\mH$-modules with grading shift functor $\langle j\rangle$ defined as $(M\langle j\rangle)_i=M_{i-j}$ for $M=\bigoplus_{i\in\mZ} M_i\in\mH\mgmod$.
\begin{prop}
\begin{itemize}
\item   $\mH\mgmod$ with $\wt(L\langle i\rangle)=-i$ and $(1)=\langle 1\rangle$ is a noetherian and artinian mixed abelian category with Tate twist.
\item Let $p(q)=\sum_{i \ge 0} (\dim \mH^i) q^i$ be the Poincar\'e polynomial of $\mH$. It has nontrivial constant term, so it can be inverted in the ring $\cR$. In fact, $[L]=p(q)^{-1}[\mH]$ in $\Kh(\cA)$, and each of $[L]$ and $[\mH]$ gives an $\cR$-basis for $\Kh(\cA)$ (which is a free $\cR$-module of rank $1$).
\end{itemize}
\end{prop}

\begin{proof} The first statement is just Example \ref{ex:mixed}. By Theorem~\ref{thm:main}\eqref{it:basis}, each of $[L]$ and $[\mH]$ gives an $\cR$-basis, since $\mH$ is local, hence has up to isomorphism and grading shift a unique simple module. The formula $[L]=p(q)^{-1}[\mH]$ follows then just by a basis transformation.
\end{proof}
A natural example arising in this context is the cohomology ring $\mH=H^*(X)$ of a smooth projective complex algebraic variety $X$. If we choose for instance $X=\mathbb{CP}^1$ then $\mH=H^*(X)=\mC[x]/(x^2)$ with Poincare polynomial $p(q)=1+q^2$, and we obtain the equation $[L]=\frac{1}{1+q^2}[\mH]=(1-q^2+q^4-q^6+\ldots)[\mH]$ in $\Kh(\cA)$.

More generally, if $\mH=H^*(X)$, where $X=\op{Gr}(i,n)$ is the Grassmannian variety of complex $i$-planes in $\mathbb{C}^n$, or any partial flag variety $X=\op{GL}(n,\mC)/P$ for some parabolic subgroup $P$, then the complex cohomology rings $H^*(X)$ are explicitly  known (see for instance \cite{Fu}, \cite{Hiller}). We have the equality $[\mH]=\binom{n}{d_1, \ldots, d_r}[L]$ in the Grothendieck group of graded $\mH$-modules, where
$$\binom{n}{d_1, \ldots, d_r} = \frac{[n]!}{[d_1] ! [d_2] ! \cdots [d_r] ! [(n-d_1
- \cdots - d_r)]!}$$
denotes the {\it quantum binomial coefficient} defined by taking the quantum numbers $[n]=\frac{q^{2n}-1}{q^2-1}=1+q^2+\cdots+q^{2(n-1)}$ for $n\in\mZ_{> 0}$ and their factorials $[n]!=[1][2][3]\cdots [n]$ with $[0]!=1$. Interpreting this quantum binomial coefficient as a formal power series in $q$, we obtain the equation
$$[L]=\frac{1}{\binom{n}{d_1, \ldots, d_r}}[\mH]$$ in $\Kh(\cA)$. By Theorem \ref{thm:main}, $L$ and $[\mH]$ each form an $\cR$-basis of $\Kh(\cA)$, and the transformation matrix is given by quantum binomial coefficients and their inverses. This transformation matrix also occurs in the representation theory of the smallest quantum group $U_q(\mathfrak{sl}_2)$, as we will see in the next section.

\subsection{Categorification of finite-dimensional irreducible modules for quantum $\mathfrak{sl}_2$}
Let $\mC(q)$ be the field of rational functions in an indeterminate $q$. Let $\mathcal{U}_q=\mathcal{U}_q(\mathfrak{sl}_2) $ be the associative algebra over $\mathbb{C}(q)$ generated by $ E, F, K, K^{-1} $ subject to the
relations:
\begin{eqnarray*}
&KK^{-1} = K^{-1}K=1, \quad KE = q^2 EK, \quad KF = q^{-2} FK,& \\
&EF-FE = \frac{K-K^{-1}}{q-q^{-1}}.&
\end{eqnarray*}

Let $ \bar{V}_n $ be the unique (up to isomorphism) irreducible module for $\mathfrak{sl}_2$ of dimension $ n+1$. Denote by $V_n $ its quantum analogue (of type I), that is
the irreducible $ \mathcal{U}_q(\mathfrak{sl}_2)$-module with basis $ \lbrace v_0, v_1, \ldots, v_{n} \rbrace $ such that
\begin{equation}
\label{irreddef}
K^{\pm 1} v_i= q^{\pm (2i-n)} v_i\quad\quad Ev_i =[i+1] q^{-i-1}v_{i+1}\quad\quad F v_i = [n-i+1] q^{1-i}v_{i-1}.
\end{equation}
Note that it is defined over $\fR$. The chosen basis is the canonical basis in Lusztig's theory of canonical bases (\cite{L}, \cite{FK}) and pairs via a bilinear form with the {\it dual canonical basis} given by $v^i=q^{-i(n-i)}\frac{1}{\binom{n}{i, n-i}}v_i$. Hence, passing to the completion $\hat{V}_n$ of $V_n$ we have an isomorphism of $\cR$-modules
\begin{eqnarray}
\hat{V}_n&\mapsto& \displaystyle\bigoplus_{i=0}^n \Kh(\cA_i)\label{eqn:iso}\\
v_i&\mapsto& [H^*(\op{Gr}(i,n))\langle i(n-i)\rangle]\nonumber\\
v^i&\mapsto& [L_i]\nonumber
\end{eqnarray}
where $\cA_i$ denotes the mixed abelian category $H^*(\op{Gr}(i,n))$-$\gmod$ with unique simple object $L_i$ of weight zero. The action of the quantum group can then be realized via correspondences: if we let $\op{Gr}(i,i+1,n)$ be the variety of partial flags
\[
(\text{$i$-plane}) \subset (\text{$(i+1)$-plane}) \subset \mC^n,
\]
then $H^*(\op{Gr}(i,i+1,n))$ is naturally a $\left(H^*(\op{Gr}(i,n)), H^*(\op{Gr}(i+1,n)\right)$-bimodule or a $\left(H^*(\op{Gr}(i+1,n)\right), H^*(\op{Gr}(i,n))$-bimodule. Tensoring (with appropriate grading shifts) with these bimodules defines exact endofunctors on $\oplus_{i=0}^n (\cA_i)$ which induce the action of $E$ and $F$ on $\oplus_{i=0}^n \Kh(\cA_i)$ given by the formula \eqref{irreddef} via the isomorphism \eqref{eqn:iso}. For details see \cite[Section 6]{FKS} and \cite{CR} in the non-graded version.
\subsection{Quotient categories}
Let $k$ be an algebraically closed field, and let $A=\bigoplus_{i\in\mZ_{\geq 0}}A_i$ be a finite-dimensional positively graded $k$-algebra, semisimple in degree zero. Let $A=\oplus_{i=1}^r Ae_i$ be the decomposition into indecomposable projective modules with simple quotients $L_i$, $1\leq i\leq r$. Let $\cA= A\mgmod$ be the mixed category of finite-dimensional graded right $A$-modules with Tate twist $(1)=\langle 1\rangle$. Assume we are given a Serre subcategory $\cS_I$ of $\cA$ stable under Tate twist. That is, $\cS_I$ is a full subcategory consisting of all modules which have composition factors only of the form $L_i\langle j\rangle$, where $j\in\mZ$ and $i\in I$ for some fixed subset $I$ of $\{1,\ldots, r\}$. Let
$$Q:\cA\rightarrow \cA/\cS_I$$
be the quotient functor to the Serre quotient $\cA/\cS_I$. Under the identification of $\cA/\cS_I$ with graded modules over $\END_A(P_I)$, where $P_I=\bigoplus_{i\notin I}Ae_i$ we have $Q=\HOM_A(\bigoplus_{i\in I}Ae_i,-)$. In particular, $Q$ is exact and has left adjoint $Q':M\mapsto M\otimes_{\END_A(P_I)}P_I$. Now the following is just a direct application of our main result:

\begin{prop}
The functors $Q:\cA\rightarrow \cA/\cS_I$ and $Q':\cA/\cS_I\rightarrow \cA$ are exact and right exact respectively, commute with Tate twist and have finite weight amplitude. Hence the functors $Q$ and $\cL Q'$
induce continuous homomorphisms of $\cR$-modules\label{prop:der-functor}
\[
[Q]: \Kh(\cA) \to \Kh(\cA/\cS_I),\quad [\cL Q']: \Kh(\cA/\cS_I) \to \Kh(\cA).
\]
\end{prop}

Note that $(\cL Q'\circ Q)^2\cong \cL Q'\circ Q$, since $Q\circ \cL Q'\cong\op{id}$. In \cite{FSS} this property is used to categorify the Jones--Wenzl projectors $\hat{V}_i\otimes \hat{V}_j\rightarrow \hat{V}_k$ for any summand $V_k$ of $V_i\otimes V_j$.


\begin{thebibliography}{BBD}

\bibitem{at:bstc}
P.~Achar and D.~Treumann, {\em Baric structures on triangulated categories and coherent sheaves}, Int. Math. Res. Not. IMRN {\bf 2011}, 3688--3743.

\bibitem{Bondarko}
M. V. Bondarko, {\em Weight structures vs. t-structures; weight filtrations,
spectral sequences, and complexes (for motives and in general)}. J. K-
Theory, {\bf 6}, (2010), no.~3, 387-–504.

\bibitem{BGS}
A. Beilinson, V.  Ginzburg and W. Soergel, {\em Koszul duality patterns in representation theory}, J. Amer. Math. Soc., {\bf 9}, no. 2, (1996), 473--527.

\bibitem{CR}
J. Chuang and R. Rouquier,
{\em Derived equivalences for symmetric groups and $\mathfrak{sl}_2$-categorification}, Ann. of Math. {\bf 167} (2008), 245--298.

\bibitem{Del}
P.  Deligne, {\em La conjecture de Weil. II}, Inst. Hautes Etudes Sci. Publ. Math., {\bf 52} (1980), 137--252.

\bibitem{FK} I. Frenkel and M. Khovanov, {\em Canonical bases in tensor products and graphical calculus for $ U_q(sl_2)$}, Duke Math J. {\bf 87} (1997), no.~3, 409--480.

\bibitem{FKS} I. Frenkel, M. Khovanov and C. Stroppel, {\em A Categorification of Finite-Dimensional Irreducible Representations of Quantum $ \mathfrak{sl}_2 $ and Their Tensor Products}, Selecta Math. (N.S.) {\bf 12} (2006), 379--431.

\bibitem{Fu} W. Fulton, {\em Young Tableaux}, London Mathematical Society Student Texts, vol.~35, Cambridge University Press, Cambridge, 1997.

\bibitem{Hiller} H. Hiller, {\em Geometry of {C}oxeter groups}, Research Notes in Math., vol.~54, Pitman, 1982.

\bibitem{KMS}
M. Khovanov, V. Mazorchuk and C. Stroppel, {\em A brief review of abelian categorifications}, Theory and Applications of Categories, {\bf 22} (2009), no.~19, 479--508.

\bibitem{FSS}
I. Frenkel, C. Stroppel and J. Sussan, {\em Categorifying fractional Euler characteristics, Jones--Wenzl projector and $3j$-symbols}, to appear in Quantum Top., arXiv:1007.4680.

\bibitem{L} G. Lusztig, {\em Canonical bases in tensor products}, Proc. Natl. Acad. Sci. USA, {\bf 89} (1992), 8177--8179.

\bibitem{Maz:lec} V. Mazorchuk, {\em Lectures on algebraic categorification}, arXiv:1011.0144.

\bibitem{miy:gguc}
J.-I. Miyachi, {\em Grothendieck groups of unbounded complexes of finitely
  generated modules}, Arch. Math. {\bf 86} (2006), 317--320.

\bibitem{Saito}
M. Saito, {\em  Mixed Hodge modules}, Publ. Res. Inst. Math. Sci., {\bf 26} (1990), no.~2, 221--333.

\bibitem{sim:shpcm}
A.-M. Simon, {\em Some homological properties of complete modules}, Math. Proc. Camb. Phil. Soc. {\bf 108} (1990), 231--246.

\bibitem{Schnuerer}
O. Schn\"urer, {\em Homotopy categories and idempotent completeness, weight structures and weight complex functors}, arXiv:1107.1227.

\bibitem{Strlccpt}
M. Stroppel, {\em Locally compact groups}, EMS Textbook 2006.
\end{thebibliography}
\end{document}